\documentclass{amsart}
\usepackage[utf8]{inputenc}
\usepackage[top=1in, bottom=1in, left=1in, right=1in]{geometry}
\usepackage{times, amsthm, amssymb, amsmath, amsfonts, graphicx, mathrsfs}
\usepackage{mathtools}
\usepackage{enumerate}
\usepackage{tikz} 
\usetikzlibrary{arrows, cd, matrix,positioning}
\usepackage{xcolor}
\usepackage{bbm}
\usepackage{xypic}
\usepackage{mathscinet}
\usepackage{algorithm}
\usepackage{algpseudocode}
\usepackage{quiver}

\newtheorem{theorem}{Theorem}[section]
\newtheorem{lemma}[theorem]{Lemma}
\newtheorem{proposition}[theorem]{Proposition}
\newtheorem{corollary}[theorem]{Corollary}

\newtheorem{question}[theorem]{Question}

\theoremstyle{definition}
\newtheorem{definition}[theorem]{Definition}
\newtheorem{example}[theorem]{Example}
\newtheorem{construction}[theorem]{Construction}

\theoremstyle{remark}
\newtheorem{remark}[theorem]{Remark}

\newtheorem*{notation}{Notation}

\DeclareMathOperator{\Ima}{Im}
\DeclareMathOperator{\Sym}{Sym}

\renewcommand{\hom}{\operatorname{Hom}}

\newcommand{\DM}{\operatorname{DM}}
\newcommand{\fold}{\operatorname{Fold}}

\newcommand{\pdim}{\operatorname{pdim}}
\newcommand{\ext}{\operatorname{Ext}}
\newcommand{\Spec}{\operatorname{Spec}}
\newcommand{\ch}{\operatorname{Ch}}
\newcommand{\pf}{\operatorname{Pf}}
\newcommand{\im}{\operatorname{im}}

\newcommand{\End}{\operatorname{End}}
\newcommand{\socdeg}{\operatorname{Socdeg}}

\newcommand{\dm}{\operatorname{DM}}

\newcommand{\bbN}{\mathbbm{N}}

\newcommand{\bbZ}{\mathbbm{Z}}
\newcommand{\bbz}{\mathbbm{Z}}

\newcommand{\bbk}{\mathbbm{k}}

\newcommand{\calX}{{\mathcal{X}}}

\newcommand{\bfF}{\mathbf{F}}

\newcommand{\del}{{\partial}}

\renewcommand{\th}{{^\text{th}}}

\newcommand{\qeq}{\overset{?}{=}}

\title{Differential modules and Deformations of Free Complexes}
\author{Maya Banks}
\author{Keller VandeBogert}
\date{\today}

\begin{document}

\maketitle

\begin{abstract}
    We classify (up to quasi-isomorphism) the free differential modules whose homology is equal to a given module $M$ by developing a theory for deforming an arbitrary free complex into a differential module. We use an iterative approach to parameterize the deformations and obstructions in terms of certain $\ext$ groups, giving an algorithmic realization of a result of Brown-Erman. We apply this theory to study certain rigidity properties of free resolutions and related rank conjectures.
\end{abstract}

\section{Introduction}

A differential module is a module equipped with a square-zero endomorphism. This square-zero map makes differential modules a natural generalizations of chain complexes---indeed, the category of differential modules is the category to which one sometimes must pass when the classical notion of $\bbz$-graded complexes is not general enough. This phenomenon has been observed for instance in the multigraded analog of Koszul duality (see \cite{hawwa2012koszul}, \cite{brown2021tate}) as well as in the literature on matrix factorizations, where it has been realized that one needs to ``deform" complexes in order to represent certain functors on the category of matrix factorizations (see \cite[Lemma 5.3.6]{oblomkov2020soergel}). Additionally, considering these more general objects can often shed new light on classical homological results by highlighting which properties of complexes are essential and which structure is extraneous. For instance, while arbitrary differential modules are quite general and lack the structure necessary to prove certain homological results, work of Avramov-Buchweitz-Iyengar \cite{avramov2007class} shows that differential modules admitting the structure of a \emph{flag} (see Definition \ref{def:flagGrading}) still possess many of the same properties as complexes. Differential modules with a flag structure are objects that lie somewhere between complexes---whose differentials respect a strict homological grading---and arbitrary differential modules. Work on the Buchsbaum-Eisenbud-Horrocks Conjecture on the Betti numbers of free resolutions and related conjectures of Halperin and Carlsson in algebraic topology \cite{avramov2007class,iyengar2018examples,walker2017total} helps bring into focus a structural spectrum ranging from minimal free resolutions as the most specific objects to arbitrary differential modules as the most general. At each step, we see certain results that carry over to the next level of generality, while others fail.



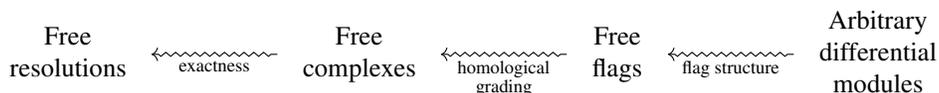
\begin{figure}[H]
\begin{tikzcd}[column sep=huge]
\begin{tabular}{c}\text{Free}\\ \text{resolutions}\end{tabular}  & \begin{tabular}{c}\text{Free}\\ \text{complexes}\end{tabular} \arrow[l,squiggly, "\text{exactness}"]  & \begin{tabular}{c}\text{Free}\\ \text{flags}\end{tabular} \arrow[l, squiggly, "\substack{\text{homological}\\\text{grading}}"] & \begin{tabular}{c}\text{Arbitrary}\\ \text{differential}\\ \text{modules}\end{tabular} \arrow[l,squiggly,"\text{flag structure}"]
\end{tikzcd}
\end{figure}

The overarching goal motivating our work is to understand how the objects on the right-hand side of the picture related to the more well-studied objects on the left. 
Brown-Erman begin to examine this relationship in their work on minimal free resolutions of differential modules. 
They demonstrate a tight link between arbitrary differential modules and free resolutions by showing that (under suitable hypotheses on the ambient ring), every differential module admits a \emph{minimal free resolution} that arises as a summand of a free flag whose structure is controlled by the minimal free resolution of the homology
(see \cite[Theorems 3.2 and 4.2]{brown2021minimal}). These results demonstrate that the theory of arbitrary differential modules can be related back to the theory of minimal free resolutions, while at the same time raising more questions about the nature of this relationship---what properties of the homology can be `lifted' to properties of the differential module? How much variation can there be in the behavior of differential modules with the same homology? What new geometric data (if any) do we add by passing from a minimal free resolution to a more general differential module with the same homology? To begin to answer these questions, we focus first on the following:

\begin{question}\label{q:dmsfixedhomology}
What are the possible differential modules with homology isomorphic to a given module? When are two such differential modules (quasi)isomorphic?
\end{question}

The key results of Brown-Erman allow us to pass from arbitrary differential modules to a special class of free flags---those that are \emph{anchored on} the minimal free resolution of their homology (see Definition \ref{def:flagGrading})---so long as we work up to quasi-isomorphism. This reduces our question to identifying the possible free flags that are anchored on the minimal free resolution of their homology. 
A free flag may be pictured as a free complex with additional compatible maps that ``go with the flow" of the complex's original differential but do not strictly respect the homological grading. The \emph{anchor} of the flag is the underlying complex, shown in red in the following visualization.
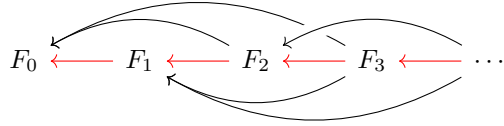
\begin{figure}[H]
\begin{tikzcd}
F_0 & F_1 \arrow[l, red] & F_2 \arrow[l, red] \arrow[ll, bend right] & F_3 \arrow[l, red] \arrow[lll, bend right] \arrow[ll, bend left] & \cdots \arrow[l, red] \arrow[ll, bend right, crossing over] \arrow[lll, bend left]
\end{tikzcd}
\end{figure}
Thus, the question at hand can be restated as follows:

\begin{question}
    If we start with a minimal free resolution, what are all of the ways that we can add in additional maps so that the total differential still squares to zero \emph{and} the homology remains unchanged?
\end{question}


Even in a small case, this question can be subtle.

\begin{example}
    Let $S = \bbk[x_1, \ldots, x_n]$ and $(D,\del_f)$ be the differential module with underlying module $D = S\oplus S(-1)^2\oplus S(-2)$ and differential given by the matrix
    \[
    \del_f = \begin{pmatrix}
    0 & x_1 & x_2 & f\\
    0 & 0 & 0 & -x_2\\
    0 & 0 & 0 & x_1\\
    0 & 0 & 0 & 0
    \end{pmatrix}
    \]
    where $f$ is some homogeneous degree 2 polynomial. For all choices of $f$, the homology of $(D,\del_f)$ is isomorphic to $S/(x_1, x_2)$ and $(D,\del_f)$ is anchored on the minimal free resolution of its homology, which we can visualize like this:
    \begin{figure}[H]
    \begin{tikzcd}[ampersand replacement=\&]
    S \& S(-1)^2 \arrow[l, red, "{\begin{pmatrix} x_1 & x_2 \end{pmatrix}}" red]  \& S(-2 \arrow[l, red, "{\begin{pmatrix} x_1 \\ x_2 \end{pmatrix}}" red] \arrow[ll, bend right, "f" '] 
    \end{tikzcd}
    \end{figure}

    It is not too hard to see that when $f\in (x_1, x_2)$, $D_f\simeq D_0$ since we can turn the matrix $\del_f$ into 
 the matrix $\del_0$ via elementary row and column operations. But when $f\notin (x_1, x_2)$ we get a differential module $D_f$ that is not isomorphic to $D_0$, meaning that we have at least two different isomorphism classes of free flags anchored on the minimal free resolution of $S/(x_1,x_2)$. Now, we ask: given two (nonzero) degree 2 polynomials $f,g$ are $D_f$ and $D_g$ isomorphic? To answer this question amounts to checking for similarity of matrices over the polynomial ring in $n$ variables, which is in general a hard problem. Note that it is easy to parameterize the set of free flags in the previous example that are anchored on the minimal free resolution of $S/(x_1, x_2)$ \emph{if we don't care about isomorphism}---such free flags are parameterized by homogeneous degree 2 polynomials. However, even in this small example it is not clear how to account for isomorphism in such a parameterization.
\end{example}

One approach to this problem was taken in \cite{banks2022differential}, where the authors construct a large family of differential modules anchored on the Koszul complex and prove that under certain hypotheses any differential module with complete intersection homology is quasi-isomorphic to one of these ``Koszul differential modules". That being said, there are examples of free flags anchored on the Koszul complex that \emph{do not} arise as a generalized Koszul differential module (see Example $4.7$ of \cite{banks2022differential}), so this is not a complete characterization. Moreover, this approach leaves open the fundamental questions about the structure theory of differential modules mentioned above.

In this paper, we present a totally different approach for answering Question \ref{q:dmsfixedhomology} that also seems to be a powerful tool for studying the uniform and asymptotic behavior of free flags with a fixed anchor. We furthermore recover and answer many of the results and questions posed in the work \cite{banks2022differential}, and further clarify the deep connection between the classical homological theory of modules and that of differential modules. Our approach is motivated by deformation theory. Brown-Erman point out that a free flag can be deformed to its anchor by adjoining a variable and zeroing out all of the maps above the anchor. We investigate the reverse direction of this observation: Given a free complex $\bfF$, we can deform $\bfF$ into a free flag by adding additional maps in a compatible way. We give a sense of the core idea below, reserving the full rigor and detail for later in the paper.

\subsection{Example: A Deformation Theoretic Approach}
Suppose we have a complex 
\[
F_0\xleftarrow{\del_{1,0}} F_1\xleftarrow{\del_{2,1}} F_2\xleftarrow{\del_{3,2}} F_3\xleftarrow{\del_{4,3}} F_4 .
\]
We consider this as a differential module with underlying module $F_0\oplus F_1\oplus F_2\oplus F_3\oplus F_4$ and differential given by the following block matrix.

\[
\begin{pmatrix}
0 & \del_{1,0} & 0 & 0 & 0\\
0 & 0 & \del_{2,1} & 0 & 0\\
0 & 0 & 0 & \del_{3,2} & 0\\
0 & 0 & 0 & 0 & \del_{4,3}\\
0 & 0 & 0 & 0 & 0
\end{pmatrix}.
\]
We want to think about deforming this complex by iteratively adding maps to the upper right blocks of this matrix. 
To do this, we adjoin a new variable $t$ to our ring and add maps $\del_{i,i-2}$ with $t$ as a coefficient to get the matrix
\[
\begin{pmatrix}
0 & \del_{1,0} & t\del_{2,0} & 0 & 0\\
0 & 0 & \del_{2,1} & t\del_{3,1} & 0\\
0 & 0 & 0 & \del_{3,2} & t\del_{4,2}\\
0 & 0 & 0 & 0 & \del_{4,3}\\
0 & 0 & 0 & 0 & 0
\end{pmatrix}.
\]
Requiring that this matrix square to 0 mod $t^2$ imposes the relations
\begin{align*}
    \del_{1,0}\del_{3,1} + \del_{2,0}\del_{3,2} &= 0\\
    \del_{2,1}\del_{4,2} + \del_{3,1}\del{4,3} &= 0
\end{align*}
which characterize the possible choices for $\del_{i,i-2}$. Next we add maps $\del_{i,i-3}$ with a coefficient of $t^2$ to get the matrix

\[
\begin{pmatrix}
0 & \del_{1,0} & t\del_{2,0} & t^2\del_{3,0} & 0\\
0 & 0 & \del_{2,1} & t\del_{3,1} & t^2\del_{4,1}\\
0 & 0 & 0 & \del_{3,2} & t\del_{4,2}\\
0 & 0 & 0 & 0 & \del_{4,3}\\
0 & 0 & 0 & 0 & 0
\end{pmatrix}.
\]
Requiring this to square to 0 mod $t^3$ imposes a new relation 
\[
\del_{1,0}\del_{4,1}+\del_{2,0}\del_{4,2} + \del_{3,0}\del_{4,3} = 0.
\]
This characterizes possible choices for $\del_{3,0}$ and $\del_{4,1}$, but it also introduces an obstruction: we may have chosen the $\del_{i,i-2}$ maps in such a way that there is \emph{no} compatible choice at the next step! For the final step, notice that any choice of map $t^3, \del_{4,0}$ in the top right corner of the matrix will be compatible with the square-zero condition mod $t^4$. Furthermore any such matrix that squares to 0 mod $t^4$ actually squares to 0 on the nose, since the square doesn't have any entries with a power of $t$ larger than 3 (in fact, no entry has a power of $t$ larger than 2). Setting $t=1$ we get a free flag, anchored on the complex we started with, whose differential is
\[
\begin{pmatrix}
0 & \del_{1,0} & \del_{2,0} & \del_{3,0} & \del_{4,0}\\
0 & 0 & \del_{2,1} & \del_{3,1} & \del_{4,1}\\
0 & 0 & 0 & \del_{3,2} & \del_{4,2}\\
0 & 0 & 0 & 0 & \del_{4,3}\\
0 & 0 & 0 & 0 & 0
\end{pmatrix}.
\]
What is more, we will see later on that the homology remains unchanged at each step, and that any free flag anchored on the minimal resolution of its homology can be obtained via a similar process. One way to understand the possibilities for such a flag is thus to understand the possible deformations and obstructions at each step in the process outlined above. Our core question now becomes  

\begin{question}\label{question:MainQuestion}
    What are all of the ways to deform a complex to a free flag? What are the obstructions to deforming, and when do these obstructions vanish?
\end{question}

\subsection{Results} The main results of this paper constitute a complete answer to Question \ref{q:dmsfixedhomology} via a deformation-theoretic approach. We state our results here for differential modules with a given homology, but note that in Section \ref{sec:constructingFreeFlag} we actually prove the following results in slightly more generality.

\begin{theorem} Let $M$ be a module over $S$ with finite minimal free resolution $\bfF$. Every differential module with homology $M$ is quasi-isomorphic to one that can be realized as a deformation of $\bfF$. Furthermore, both the deformations of $\bfF$ and the obstructions to deforming are represented by classes in $\ext^\bullet_S(M,M)$.

\end{theorem}

When we work over an algebraically closed field, this description of the deformations and obstructions yields a simple geometric description of the quasi-isomorphism classes of differential modules with homology $M$.

\begin{theorem}
Assume $\bbk$ is algebraically closed and let $\calX^M_a$ denote the set of degree $a$ differential modules with homology $M$, up to quasi-isomorphism. Then $\calX^M_a$ has the structure of an algebraic variety over $\bbk$ whose dimension is bounded above by
\[
\sum_{i=2}^\ell \dim_\bbk \ext_S^i (M,M)_{a-ia}
\]
and below by
\[
\sum_{i=2}^\ell\dim_\bbk \ext_S^i (M,M)_{a-ia} - \sum_{i=4}^\ell\dim_\bbk \ext_S^i (M,M)_{2a-ia}.
\]   
\end{theorem}

\begin{corollary}
Assume $\bbk$ is algebraically closed and $M$ is any $S$-module with minimal free resolution $\bfF$. If $$\ext^i_S(M,M)_{a-ia} = 0$$ for all $2\leq i\leq\pdim(M)$, then every degree $a$ differential module with homology $M$ is quasi-isomorphic to $\fold^a(\bfF)$ (in other words, $M$ is \emph{a-rigid}; see Definition \ref{def:aRigid}).
\end{corollary}

Note that it is easy to write down a set of equations that cuts out the set of all free flags anchored on a fixed complex using only the condition that the differential squares to zero, but it is not at all clear how to write down a set of equations that also accounts for isomorphism class. Indeed, if we were to naively take the variety cut out by these equations and try to just mod out by isomorphism of differential modules, there is absolutely no indication, a priori, that the resulting object would be a variety. Our approach sidesteps this particular challenge by constructing $\calX^M_a$ iteratively as a sequence of deformations, then showing that the isomorphism class of two such deformations at each step is controlled by higher Ext classes of the module $M$. 

Our approach also yields several applications, one of which is that, for all but finitely many values of $a$, the study of degree $a$ differential modules with finite length homology reduces to the study of minimal free resolutions of finite length $S$-modules. In particular:

\begin{theorem}
    Let $D$ be a free degree $a$ differential module with finite length homology $M$. Then for all $|a|\gg 0$, $D$ is quasi-isomorphic to the minimal free resolution of $M$.
\end{theorem}

We prove this, as well as other applications and examples, in Section \ref{sec:Applications}. More generally, we enumerate a list of conditions that is sufficient for the space parameterizing the isomorphism classes of free flags anchored on a fixed complex to consist of a single point, and in the process recover (and generalize) previous results proved in \cite{banks2022differential} with a completely different method. Finally, we discuss some applications of our work to rank conjectures for differential modules. We give some results about the existence of free flag differential modules whose total Betti number is strictly smaller than the sum of the Betti numbers of the homology; as it turns out, our characterization of the ``deformation" terms gives a straightforward criterion in special cases for checking whether a differential module with unexpectedly small Betti numbers exists.

The paper is organized as follows. In Section \ref{sec:backgroundFlags}, we establish the necessary background and conventions on (free flag) differential modules that will be essential for the rest of the paper. In Section \ref{sec:constructingFreeFlag} we consider an iterative process of constructing free flags with a fixed anchor and connect this process to the (non)triviality of certain cohomology classes of the associated endomorphism complex. After building up the necessary machinery, we prove our main result of the paper, Theorem \ref{thm:varietyOfIsoClasses}, and note some interesting consequences for the derived category of differential modules. In Section \ref{sec:Applications}, we apply the results established in Section \ref{sec:constructingFreeFlag} to understand multiple different aspects of free flag differential modules, including rigidity, Betti deficiency, and the connection between higher structure maps and the notion of systems of higher homotopies for matrix factorizations.

\section*{Acknowledgments}

We thank Daniel Erman and Michael Brown for many helpful discussions regarding this work. We further thank Daniel Erman for his comments and suggestions on earlier drafts of this paper. VandeBogert acknowledges the support of the National Science Foundation Grant DMS-2202871.

\section{Anchored free flags}\label{sec:backgroundFlags}

In this section, we introduce some background and notation on differential modules that will be used throughout the paper. From now on, unless otherwise specified, the notation $S$ will denote any Noetherian graded-local ring, by which we mean an $\bbN$-graded ring (possibly concentrated in degree 0) whose degree 0 component is local with residue field $\bbk$. All modules considered throughout this paper will be finitely generated.

Let us first recall the notion of a \emph{homogeneous} differential module:

\begin{definition}\label{def:diffMod}
A \emph{differential module} $(D,d)$ or $(D,d^D)$ is an $S$-module $D$ equipped with an $S$-endomorphism $d = d^D \colon D \to D$ that squares to $0$. A differential module is graded of (internal) degree $a$ if $D$ is equipped with a grading over $S$ such that $d \colon D \to D(a)$ is a graded map. The category of degree $\mathbf{a}$ differential $S$-modules will be denoted $\dm (S,a)$.

The \emph{homology} of a differential module $(D,d)$ is defined to be $\ker (d) / \im (d)$. If $D$ is graded of degree $a$, then the homology is defined to be the quotient $\ker(d) / \im(d(- a))$.

A differential module is \emph{free} if the underlying module $D$ is a free $S$-module, and  $D$ is \emph{minimal} if $d \otimes_S \bbk = 0$ (that is, the differential module $D$ is minimal if its squarezero endomorphism is minimal).

A morphism of differential modules $\phi \colon (D,d^D) \to (D', d^{D'})$ is a morphism of $S$-modules $D \to D'$ satisfying $d^{D'} \circ \phi = \phi \circ d^D$. Notice that morphisms of differential modules induce well-defined maps on homology in an identical fashion to the case of complexes. A morphism of differential modules is a \emph{quasi-isomorphism} if the induced map on homology is a quasi-isomorphism.
\end{definition}

\begin{example}
    Every homogeneous complex of $S$-modules may be viewed as a degree $0$ differential module.
\end{example}

\begin{remark}
    Differential modules are essentially the homologically ungraded analog of complexes, and as such they can vary much more wildly than complexes. The useful analogy here is the difference between graded rings versus ungraded rings: a differential $S$-module is equivalently a $S[t]/(t^2)$-module; on the other hand, the ring $S[t]/(t^2)$ may be viewed as a $\bbz$-graded ring by assigning $\deg (t) = 1$. Complexes are thus equivalently described as $\bbz$-graded $S[t]/(t^2)$-modules with respect to this grading, and homogeneous complexes are equivalently graded $S[t] / (t^2)$-modules where $t$ is assigned the bidegree $(1,0)$ (the first component denotes homological degree, and the second component internal degree).
\end{remark}

\begin{definition}
    Let $S$ be a (positively) graded ring. The \emph{degree $a$ fold} is the functor:
    \begin{align*}
\fold_a\colon &\ch(S)\to \DM(S,a)\\
& (C_\bullet, \del)\mapsto \left(\bigoplus C_i(ia), \del \right),
\end{align*}
where in the above $\ch (S)$ is the category of homogeneous complexes of degree $0$.
\end{definition}

\begin{remark}\label{rem:embeddingCh(S)}
    Notice that by construction the degree $a$ fold of a homogeneous complex yields a differential module of degree $a$. In particular, the category of homogeneous chain complexes of $S$-modules not only embeds into the category of differential modules, but it also embeds into $\DM (S,a)$ for all $a \in \bbz$. On the other hand, there is no clear relation between the categories $\DM (S,a)$ and $\DM (S,b)$ for different values of $a$ and $b$; this is at least indicated by the fact that for a fixed module $M$, the Betti numbers of a differential module $D \in \DM (S,a)$ with $H(D) = M$ may change as $a$ varies (see \cite[Proposition 3.8]{banks2022differential}). 
\end{remark}

Next, we define the notion of a \emph{free flag}. These can be thought of as a useful intermediary between chain complexes of free $S$-modules and general differential modules, and as we shall see are an important ingredient for generalizing free resolutions to the category of differential modules.

\begin{definition}\label{def:flagGrading}
Let $D$ be a differential module. Then $D$ is a \emph{free flag} if $D$ admits a splitting $D = \bigoplus_{i \in \bbz} F_i$, where each $F_i$ is a free $S$-module, $F_i = 0$ for $i \leq 0$, and $d_D (F_i) \subseteq \bigoplus_{j < i} F_j$. 

Every free flag is naturally filtered by its flag grading $D^0 = D_0 \subset D^1 \subset D^2 \subset \cdots$ by defining 
$$D^i := \bigoplus_{j \leq i} F_j.$$
The spectral sequence associated to this filtration has $E^0$-page given by a complex with differentials $d^F_i : D^i / D^{i-1} = F_i \to D^{i-1}/D^{i-2} = F_{i-1}$. This complex is called the \emph{anchor} of the free flag $D$.
\end{definition}

\begin{remark}
The squarezero endomorphism of the free flag may be represented as a strictly upper triangular block matrix, where the first off-diagonal blocks list the differentials of the anchor.
\end{remark}

The next proposition shows that ``deforming" a resolution does not change homology, a fact that we will take advantage of later when studying quasi-isomorphism classes of differential modules.

\begin{proposition}\label{prop:flagsAnchoredOnRes}
Let $D$ be a free flag and assume that the anchor of $D$ is a free resolution of some module $M$. Then $H(D) \cong M$. 
\end{proposition}

\begin{proof}
Define
$$H ( D )^i := \Ima \left( H (D^i) \to H(D) \right).$$
By definition there is an equality
$$H(D) = \bigcup_{i \geq 0} H( D )^i$$
and by \cite[Section 2.6]{avramov2007class} a spectral sequence
$$E^1_i = H_i (F) \implies H(D ),$$
where $F$ denotes the underlying anchor of $D$. Since $E^1_i = 0$ for $i \neq 0$ by assumption, the spectral sequence degenerates on the first page and there are equalities
$$H (D )^0 = H_0 (F), \quad H (D )^i = H (D )^{i-1} \quad \textrm{for} \ i >0.$$
It follows that $H(D ) = H_0 (F) = M$.
\end{proof}

\begin{remark}
More generally, the above shows that if a free flag $D$ is anchored on a complex with finite length homology, then $H(D)$ is a module of finite length (really, it suffices for any page of the associated spectral sequence to eventually have finite length homology; this is pointed out in \cite{avramov2007class}).
\end{remark}

To finish this section, we state a fundamental result of Brown-Erman. One way to interpret this statement is that if one is willing to work up to quasi-isomorphism, then the homological theory of differential modules is anchored on the classical theory of free resolutions of modules.

\begin{theorem}[{\cite[Theorem 3.2]{brown2021minimal}}]\label{thm:danMikeThm}
Let $D$ be a differential module and $(F_\bullet , d) \to H(D)$ a minimal free resolution of $H(D)$. Then $D$ admits a free flag resolution anchored on the complex $F$.
\end{theorem}

\begin{notation}
    For convenience, we write $\delta_j$ for the map $F\to F$ that, when restricted to $F_i$ is equal to $\del_{i,i-j}$ for each $i$ (in general, we assume that $F_i = 0$ for $i<0$ or $i>\ell$ so any maps to or from such $F_i$ are defined to be the zero map).
\end{notation}

\begin{remark}\label{rmk:degree}
    Suppose that $\bfF_\bullet$ is a graded complex. Let $\bfF^{a}_\bullet$ be the complex obtained by twisting $F_i$ by $ia$ so that all maps are homogeneous of degree $a$. A free flag anchored on $\bfF^{a}_\bullet$ has underlying module $\bigoplus_i F_i(ai)$ and differential $\del$ of degree $a$. For each $i>j$, $\del$ restricts to a degree $a$ map $\del_{i,j}: F_i(ia)\to F_j(ja)$. It will sometimes be helpful to remember that a degree $a$ map $F_i(ia)\to F_j(ja)$ is the equivalently a degree $a(j-i+1) $ map $F_i\to F_j$.
\end{remark}

\section{Parameterizing Anchored Free Flags}\label{sec:constructingFreeFlag}

\subsection{Building Free flags}\label{subsec:buildFreeFlag}

We describe a process for building an anchored free flag with a fixed homology $M$ and prove that all free flags anchored on the minimal free resolution of $M$ arise in this way.

To begin, we introduce the first main character of this construction, the \emph{endomorphism complex}. We will find that all of the higher differentials of a free flag (and their obstructions) may be reinterpreted as elements of this complex.

\begin{definition}
    Let $\bfF$ be a complex of $S$-modules (indexed homologically). The \emph{endomorphism complex} $\End_S^\bullet ( \bfF)$ is the (cochain) complex with
    $$\End_S^i (\bfF ) := \prod_{n \in \bbz} \hom_S (\bfF_n , \bfF_{n-i}), \quad \textrm{and differential}$$
    $$d^{\End_S^\bullet ( \bfF)} (\phi) := d^\bfF \circ \phi - (-1)^{|\phi|} \phi \circ d^{\bfF}.$$
    Equipped with function composition, the complex $\End_S^\bullet ( \bfF)$ is naturally an associative DG-algebra. It is useful to note that
    $$Z^i (\End_S^\bullet (  \bfF)) = \hom_{\ch (S)} (\bfF , \bfF [-i]),$$
    $$B^i (\End_S^\bullet (  \bfF)) = \textrm{nullhomotopies} \ \bfF \to \bfF [-i].$$
    If $\bfF$ is a free resolution of some $S$-module $M$, there is an isomorphism
    $$H^i (\End_S^\bullet ( \bfF)) = \ext^i_S (M,M).$$
\end{definition}

The following notation will be tacitly used for the remainder of this section:

\begin{notation}
Let $t$ denote any arbitrary indeterminate. The notation $S_i$ will denote the ring
    $$S_i := \frac{S[t]}{(t^i)}.$$
\end{notation}

The following construction seeks to make precise the idea of iteratively deforming a complex $\bfF$ to obtain a free flag anchored on $\bfF$. In doing this construction, one needs to work in the \emph{truncated} polynomial ring $S_i$ where $i$ increases at each step, since this will allow us to guarantee that the resulting differential module still has squarezero endomorphism.

\begin{construction}\label{const:aff}
Let $\bfF$ be any homogeneous complex of free $S$-modules
\[
\bfF : \quad  F_0\xleftarrow{\del_{1,0}} F_1 \xleftarrow{\del_{2,1}}\cdots \xleftarrow{\del_{\ell, \ell-1}}F_\ell
\]
with $(F, d_0) = \fold_a(\bfF)$.

One can iteratively build a free flag $(F,\del)$ anchored on $\bfF$ by filling in the matrix $\del$ diagonal by diagonal, at the $i^{th}$ step lifting from a differential module over $S_{i}$ to one over $S_{i+1}$.

Suppose we have built up a differential module $(F\otimes S_i, d_{i-1})\in \DM(S_i,a)$ anchored on $\bfF\otimes S_i$ where 
\[
d_{i-1} = \begin{pmatrix}
0 & \del_{1,0} & t\del_{2,0} & \cdots & t^{i-1}\del_{i, 0} & 0 & \cdots & 0\\
0 & 0 & \del_{2,1} & t\del_{3,1} & \cdots & t^{i-1}\del_{i+1, 1} & \cdots & 0\\
0 & 0 & 0 & \del_{3,2} & t\del_{4,2}& \cdots & \ddots & 0\\
0 & 0 & 0 & 0 & \del_{4,3} & \ddots & \cdots & \vdots\\
\vdots & \vdots & \vdots& \vdots & \ddots & \ddots & \ddots & \vdots\\
0 & 0 & 0 & 0 & 0 & \cdots  & \del_{\ell-1, \ell-2} & t\del_{\ell, \ell-2}\\
0 & 0 & 0 & 0 & 0 & \cdots & 0 & \del_{\ell,\ell-1}\\
0 & 0 & 0 & 0 & 0 & \cdots & 0 & 0
\end{pmatrix}
\]
If $(F\otimes_S S_i, d_{i-1})$ can be lifted to a differential module $(F\otimes_S S_{i+1}, d_i)$, we then lift by adding maps $\{t^i\del_{j+i+1, j}\}_{j=0}^{\ell-i-1}$ to $d_{i-1}$ above (adding on the next off-diagonal after the $t^{i-1}$ terms). By the $\ell^{th}$ step we have a differential module $(F\otimes_S S_\ell, d_{\ell-1})\in \DM(S_\ell, a)$. By setting $t=1$, we obtain a differential module in $(F, \del)\in \DM(S,a)$ anchored on $\bfF^a$.
\end{construction}

\begin{remark}
    With notation as in the above construction, assign the variable $t$ a homological grading by setting $\deg (t) = 1$ (and hence cohomological degree $-1$). The maps $\delta_j$ are a priori elements of $\End^j ( \bfF)$, and the process of rescaling yields the map $t^{j-1} \delta_j \in \End^1_{S_i} ( \bfF \otimes_{S} S_i)$. In other words, this rescaling process forces all of the maps to be homogeneous of cohomological degree $1$ and thus induces a $\bbz$-grading on the differential module $(F \otimes_S S_i , d_{i-1})$. This $\bbz$-grading in fact agrees with the $\bbz$-grading induced by the flag grading, but the introduction of the $t$-variable makes the bookkeeping much simpler. Similar rescaling tricks have also been used to induce $\bbz$-gradings on $\bbz / 2 \bbz$-graded complexes (see for instance the discussion on page 2182 of \cite{brown2017adams}), and is intimately related to the idea of ``unfolding" a $\bbz / 2 \bbz$-graded object.
\end{remark}

The following lemma gives an explicit obstruction to the iterative process outlined in Construction \ref{const:aff}, and these obstructions may be reformulated as nontrivial cohomology classes of the endomorphism complex:

    \begin{lemma}\label{lem:obstructions}
    The free flag $(F\otimes S_i, d_{i-1})$ can be lifted to a free flag $(F\otimes S_{i+1}, d_i)$ by adding maps $\{t^i\del_{j+i+1, j}\}_{j=0}^{\ell-i-1}$ if and only if a certain map $\bfF\to \bfF[-i-2]$ is nullhomotopic (in other words, a certain class in $H^{i+2}(\End_S^\bullet (\bfF))$ is 0).
    \end{lemma}

        \begin{proof}
        
        Let $d_i$ be the map obtained by adding the maps $t^i\delta_{i+1}$ to $d_{i-1}$ as in Construction \ref{const:aff}. This gives a differential module over $F\otimes_S S_{i+1}$ if and only if $d_i^2 = 0 \mod (t^i)$. Assuming that $d_{i-1}^2 = 0$ mod $t^{i-1}$, this happens exactly when the following condition is satisfied:
        \[
        \sum_{j=1}^{i+1}\delta_{j}\delta_{i+2-j} = 0
        \]
        This can rewritten as
        \begin{equation}\label{eq:obstruction}
        \delta_{1}\delta_{i+1}+\delta_{i+1}\delta_{1} = -\sum_{j=2}^{i} \delta_{j}\delta_{i+2-j}
        \end{equation}
        
        We claim that the right hand side of (\ref{eq:obstruction}) defines a chain map $\bfF\to-\bfF[-i-2]$. This means that one must verify that there is an equality
        
        \[
        \left( \sum_{j=2}^{i} \delta_{j}\delta_{i+2-j} \right)\delta_{1} \qeq \delta_{1} \left(\sum_{j=2}^{i} \delta_{j}\delta_{i+2-j}\right).
        \]
        By assumption $(F\otimes S_i, d_{i-1})$ is a differential module, which means that for any $2\leq p\leq i$ there are equalities
        \begin{align*}
        \delta_{p}\delta_{1} = -\sum_{k=1}^{p-1}\delta_{k}\delta_{p+1-k} \quad \text{ and } \quad \delta_1\delta_p = -\sum_{\ell=2}^{p}\delta_\ell\delta_{p+1-\ell}.
        \end{align*}
        Using this, we rewrite each of the sums above to get 
        \[
        \sum_{j=2}^{i}\sum_{k=1}^{i-j+1} \delta_{j}\delta_{k}\delta_{i+3-j-k} \qeq \sum_{j=2}^{i}\sum_{\ell=2}^j \delta_{\ell}\delta_{j+1-\ell}\delta_{i+2-j}.
        \]
        Switching the order of the sums on the right, we find
        \[
         \sum_{j=2}^{i}\sum_{k=1}^{i-j+1} \delta_{j}\delta_{k}\delta_{i+3-j-k} \qeq \sum_{\ell=2}^{i}\sum_{j=\ell}^{i} \delta_{\ell}\delta_{j+1-\ell}\delta_{i+2-j}.
        \]
        Finally, reindex the righthand sum by first replacing $j$ by $k+\ell-1$ to obtain
        \[
        \sum_{j=2}^{i}\sum_{k=1}^{i-j+1} \delta_{j}\delta_{k}\delta_{i+3-j-k} \qeq 
        \sum_{\ell=2}^{i}\sum_{k=1}^{i-\ell+1} \delta_{\ell}\delta_{k}\delta_{i+3-\ell-k}.
        \]
        Replacing $\ell$ on the right with $j$, we see that this is indeed an equality, so the right hand side of (\ref{eq:obstruction}) is a chain map as claimed.
        This means that (\ref{eq:obstruction}) is exactly the condition that the the maps $\{\del_{j,j-i},\}$ are a nullhomotopy for this map $\bfF\to\bfF[-i-2]$. This chain map corresponds to a class in $H^{i+2} ( \End_S^\bullet ( \bfF))$. By Remark \ref{rmk:degree}, the map $\delta_j$ has degree $a-ja$, so the chain map in (\ref{eq:obstruction}) has degree $2a-(i+j)a$ when considering the cohomology as a graded $S$-module.
        \end{proof}

        \begin{remark}
            Viewing the endomorphism complex $A := \End^\bullet_S (F)$ as a DG $S$-algebra in the standard way, the condition (\ref{eq:obstruction}) is (up to some sign changes which we will not get into) equivalent to the condition
            $$d^A (\delta_{i+1}) = \sum_{j=2}^{i} (-1)^{j+1} \delta_j \delta_{i+2-j}.$$
            This means that in the cohomology algebra $H^\bullet (A)$, all of the classes $\sum_{j=2}^{i} (-1)^{j+1} \delta_j \delta_{i+2-j}$ are trivial, and hence the Massey powers (in the sense of \cite[Section 3]{kraines1966massey}) $\langle \delta_2 \rangle^k$ are trivial for all $k \geq 2$. Since Massey powers are known to give well-defined cohomology classes, this yields an alternative proof that the expression on the righthand side of (\ref{eq:obstruction}) is a morphism of complexes.
        \end{remark}

In a similar vein, the following lemma shows that the isomorphism classes of two different lifts as in Construction \ref{const:aff} are also parametrized by cohomology classes of $\End_S (\bfF)$:

    \begin{lemma}\label{lem:param} 
    If lifts exist as in Lemma \ref{lem:obstructions}, then the possible lifts are parameterized by $\langle t^i\rangle H^{i+1}(\End_{S_{i+1}}^\bullet ( \bfF \otimes_S S_{i+1}))_{a-(i+1)a}$.
 
    \end{lemma}
    \begin{proof}
    The statement will follow by proving
    \begin{enumerate}[(1)]
        \item the difference of any two lifts corresponds to a well-defined cohomology class in $\End_S ( \bfF)$, and
        \item if this difference is $0$ in cohomology (that is, the maps are homotopic), then the two lifts give isomorphic differential modules.
    \end{enumerate}
    Suppose we have $d_i$ and $d'_i$ that are both lifts of $d_{i-1}$ via maps $\{t^i\del_{j+i+1, j}\}_{j=0}^{\ell-i-1}$ and $\{t^i\del'_{j+i+1, j}\}_{j=0}^{\ell-i-1}$ respectively. Since both satisfy (\ref{eq:obstruction}), their difference $\{t^i\del_{j+i+1, j} - t^i\del'_{j+i+1,j}\}_{j=0}^{\ell-i-1}$ defines a chain map $$\bfF\otimes S_{i+1}\to -\bfF[-i-1]\otimes S_{i+1}$$ which by Remark \ref{rmk:degree} belongs to a class in $H^{i+1} ( \End_{S_{i+1}}^\bullet ( \bfF\otimes_S S_{i+1}))_{a-(i+1)a}$. 
    Suppose this map is nullhomotopic. Then there exist $h_{j+i,j}: F_j\otimes S_{i+1}\to F_{j-i}\otimes S_{i+1}$ for each $j$ so that
    \begin{equation}\label{eq:homotopy-iso}
    t^i\del_{j+i+1, j} - t^i\del'_{j+i+1,j} = t^i h_{j+i,j}\del_{j+i+1,j+i}-t^i\del_{j+1,j}h_{j+i+1,j+1}.
    \end{equation}
    Let $P_i$ is the matrix with 1s on the diagonal and whose entry in the $p^{th}$ row and $q^{th}$ column is $t^ih_{q,p}$ when $q-p = i$ and 0 otherwise (for instance see $P_2$ below). 

    \[
P_2 = \begin{pmatrix}
    1 & 0 & t^2 h_{2,0} & 0  & \hdots & 0 & 0\\
    0 & 1 & 0 & t^2 h_{3,1}  & 0 & \hdots & 0\\
    0 & 0 & 1 & 0 & t^2 h_{4,2} & \hdots & 0\\
    \vdots & \vdots & \vdots & \ddots  & \ddots & \ddots & \vdots \\
   0 & 0 & 0 & \hdots  & 1 & 0 & t^2 h_{\ell, \ell-2}
\end{pmatrix}
    \]
By solving for $t^i\del'_{j+i+1,j}$ in (\ref{eq:homotopy-iso}), we can write the differential $d'_i$ as the matrix whose entries are 
\[
(d'_i)_{p,r} = \begin{cases}
t^{r-p-1}\del_{r,p} - t^{r-p}h_{r-1,p}\del_{r,r-1}+ t^{r-p}\del_{p+1,p}h_{r,p+1} & \text{ if }~~ r-p = i+1\\
t^{r-p-1}\del_{r,p} & \text{ if }~~ 1\leq r-p\leq i \\
0 & \text{ otherwise }
\end{cases}
\]
We claim that $d'_i = P_i^{-1}d_iP_i$ (mod $t^{i+1}$). 
First, notice that (mod $t^{i+1}$) the inverse $P_i^{-1}$ is the matrix whose entries are 
\[
(P_i)^{-1}_{p,k} = \begin{cases}
-t^ih_{p,k} & \text{ if }~~ p-k = i\\
1 & \text{ if }~~ p-k = 0\\
0 & \text{ otherwise }
\end{cases}.
\]
Using this, we compute $P_i^{-1}d_iP_i$ entry-wise:
\begin{align*}
    (P_i^{-1}d_iP_i)_{p,r} &= \sum_{q=0}^\ell \sum_{k=0}^\ell (P_i^{-1})_{p,k}(d_i)_{k,q}(P_i)_{q,r}\\
    &= \sum_{k=0}^\ell (P_i^{-1})_{p,k}(d_i)_{k,r} + \sum_{k=0}^\ell (P_i^{-1})_{p,k}(d_i)_{k,r-i}t^ih_{r,r-i}\\
    &= (d_i)_{p,r} + t^i(d_i)_{p,r-i}h_{r,r-i} - t^ih_{p+i,p}(d_i)_{p+i,r} - t^{2i}h_{p+i,p}(d_i)_{p+i,r-i}h_{r,r-i}\\
    &=(d_i)_{p,r} + t^i(d_i)_{p,r-i}h_{r,r-i} - t^ih_{p+i,p}(d_i)_{p+i,r}~~~(\text{mod } t^{i+1})\\
    &= \begin{cases}
    t^i\del_{r,p} + t^i\del_{r-i,p}h_{r,r-i} - t^ih_{p+i,p}\del_{r,p+i} &\text{ if }~~ r-p = i+1\\
     t^{r-p-1}\del_{r,p} & \text{ if }~~ 1\leq r-p\leq i\\
    0 & \text{ otherwise. }
    \end{cases}
\end{align*}
Since $d'_i = P_i^{-1}d_iP_i$ (mod $t^{i+1}$), it follows that $(F\otimes_S S_{i+1}, d_i)$ and $(F\otimes_S S_{i+1}, d'_i)$ are isomorphic differential modules over $S_{i+1}$. Hence, the choices of lifts from $S_i$ to $S_{i+1}$ are parameterized up to isomorphism by a class in $\langle t^i\rangle H^{i+1}(\End_S ( \bfF \otimes_S S_{i+1}))_{a-(i+1)a}$.
\end{proof}

\begin{lemma}\label{lem:ext}
There is an isomorphism of $\bbk$-vector spaces
$$\langle t^i\rangle H^{i+1} (\End_{S_{i+1}}^\bullet ( \bfF \otimes_S S_{i+1}))_{a-(i+1)a}\simeq H^{i+1} (\End_S^\bullet (\bfF))_{a-(i+1)a}.$$
\end{lemma}
\begin{proof}
Let $\phi: S\to S_{i+1}$ be the natural inclusion of $S$-algebras; notice that $S_{i+1} \cong \bigoplus_{j=0}^{i} S \cdot t^j$ is a free $S$-module. In particular, the ideal $\langle t^i \rangle := (t^i)/(t^{i+1}) \subset S_{i+1}$ is isomorphic to $S t^i$ and is hence a free $S$-module of rank $1$. Since tensoring with a free $S$-module is exact and $\bfF$ is a complex of free $S$-modules, there is an isomorphism of $S$-modules
\[
H^{i+1} (\End_{S_{i+1}}^\bullet ( \bfF \otimes_S S_{i+1}))\cong H^{i+1} ( \End_S^\bullet( \bfF) ) \otimes_S S_{i+1}.
\]
Multiplying the above by the ideal $\langle t^i \rangle \subset S_{i+1}$, it follows that
\begin{align*}
    \langle t^i\rangle H^{i+1} (\End_{S_{i+1}}^\bullet ( \bfF \otimes_S S_{i+1})) &\simeq H^{i+1} (\End_S^\bullet (\bfF)) \otimes_S \underbrace{\langle t^i \rangle S_{i+1}}_{\simeq S} \\
    &\simeq H^{i+1} (\End_S^\bullet (\bfF)).
\end{align*}
Restricting to homogeneous components yields the statement of the lemma.
\end{proof}

The content of the above lemmas combines to show that the lifts and obstructions appearing in Construction \ref{const:aff} are parameterized by cohomology classes of the endomorphism complex. The final lemma we need in order to parameterize the entire set of free flags anchored on a fixed complex $\bfF$ is the following.

\begin{lemma}\label{lem:build-aff}
Every free flag anchored on $\bfF$ may be obtained via Construction \ref{const:aff}.
\end{lemma}

\begin{proof}
Let $\bfF = ~~ F_0\xleftarrow{\del_1,0} F_1\xleftarrow{\del_{2,1}}\cdots \xleftarrow{\del_{\ell, \ell-1}}F_\ell \leftarrow 0$ be any complex of free $S$-modules. An anchored free flag $F$ on $\bfF$ has differential
\[
\del = \begin{pmatrix}
0 &\del_{1,0} & \del_{2,0} & \cdots & \del_{\ell,0}\\
0 & 0 & \del_{2,1} & \cdots & \del_{\ell,1}\\
\vdots & \vdots & \vdots & \ddots &\vdots\\
0 & 0 & 0 & \cdots & \del_{\ell, \ell-1} \\
0 & 0 & 0 & \cdots & 0
\end{pmatrix}
\]
Rescaling appropriately by powers of $t$ to $\del$, we obtain a differential module over $S[t]$ whose differential is 

\[
\del' = \begin{pmatrix}
0 &\del_{1,0} & t\del_{2,0} & \cdots & t^{\ell-1}\del_{\ell,0}\\
0 & 0 & \del_{2,1} & \cdots & t^{\ell-2}\del_{\ell,1}\\
\vdots & \vdots & \vdots & \ddots &\vdots\\
0 & 0 & 0 & \cdots & \del_{\ell, \ell-1} \\
0 & 0 & 0 & \cdots & 0
\end{pmatrix}
\]
and which specializes to $(F,\del)$ when $t=1$. Note that for each $i$, $(F\otimes S_i, \del')$ is a differential module over $S_i$, so this is the differential module we end up with if we run Construction \ref{const:aff} and at the $i\th$ step pick the maps $\del_{j,j-i-1}$ as our augmenting maps.
\end{proof}

\subsection{Dimension Bounds}


Using the material proved in \ref{subsec:buildFreeFlag}, we can show that for a fixed anchor $\bfF$, the isomorphism classes of degree $a$ free flags anchored on $\bfF$ are parametrized by a well-behaved geometric object:

\begin{theorem}\label{thm:varietyOfIsoClasses}
Let $S$ be a graded-local ring where the residue field $\bbk$ of $S_0$ is algebraically closed. The set of isomorphism classes of free flags in $\DM(S,a)$ anchored on a free complex $\bfF$ is parameterized by a nonempty variety $\calX^\bfF_a$ whose dimension satisfies the following bounds:
\[
\sum_{i=2}^\ell\dim_\bbk H^i (\End_S^\bullet (\bfF))_{a-ia} - \sum_{i=4}^\ell\dim_\bbk H^i (\End_S^\bullet (\bfF))_{2a-ia}\leq \dim \calX^\bfF_a \leq \sum_{i=2}^\ell \dim_\bbk H^i (\End_S^\bullet (\bfF))_{a-ia}.
\]
Moreover, $\dim \calX^\bfF_a = 0$ if and only if $\calX^\bfF_a$ consists of a single point corresponding to the degree $a$ fold $\fold^a (\bfF)$.


\end{theorem}

\begin{proof}
Throughout the proof, we will freely employ notation established in subsection \ref{subsec:buildFreeFlag}, particularly Construction \ref{const:aff}. The proof follows by inductively defining a family of algebraic varieties $\{X_i\}_{i \geq 1}$ using the constructions of \ref{subsec:buildFreeFlag} whose points will eventually correspond to all possible isomorphism classes of free flags anchored on $\bfF$. When $i = 1$, the variety $X_1$ is defined to be the singleton corresponding to $\fold^a (\bfF)$. 

For conciseness of notation, define $E^i_j := \operatorname{Cl} \left( \Spec(\Sym^*(H^{i}(\End_S (M))_j)) \right)$ (where $\operatorname{Cl} (-)$ denotes the functor sending a $\bbk$-scheme to its closed points). Then for each $i \geq 1$, let
\[
\nu_i: X_i\to E^{i+2}_{2a-ia},
\]
be the map that sends the point corresponding to $F \otimes_S S_i$ to the chain map on the right hand side of equation (\ref{eq:obstruction}). Since this map is polynomial in the coordinates of the underlying vector spaces, this is a morphism of (classical) varieties, and since $\bbk$ is algebraically closed this induces a morphism of $\bbk$-schemes. 

By Lemma \ref{lem:obstructions}, the subset of $X_i$ that can be lifted to a differential module in $X_{i+1}$ is the fiber $K_i := \nu_i^{-1} (0)$ above the point corresponding to $0$ in $E^{i+2}_{2a-ia}$. On the other hand, Lemma \ref{lem:param} implies that the isomorphism classes of two lifts of the free flag $(F \otimes_S S_i , d_i)$ are parameterized by $E^{i+1}_{a-(i+1)a}$, in which case we set $X_{i+1} := K_i\times E^{i+1}_{a-(i+1)a}$ for $i \geq 1$. Notice that the points of $X_{i+1}$ may explicitly be described as isomorphism classes of free flags of the form $(F \otimes_S S_{i+1} , d_{i+1})$, obtained as in Construction \ref{const:aff}. Inductively, this shows that each $X_i$ is a variety and that there is an inclusion
$$X_i \subset X_1 \times E^{2}_{-a} \times E^2_{-2a} \times \cdots \times E^i_{a-ia}.$$
Define $\calX^\bfF_a$ to be $X_\ell$ (recall that $\ell$ is the length of $\bfF$). Setting $i = \ell$ in the above inclusion and counting dimensions yields the upper bound on $\dim \calX^\bfF_a$.

Next, we will establish the lower bound on $\dim \calX^\bfF_a$. Define $V^{i+2} := \im \nu_i$. By the fiber dimension formula (see, for instance, \cite[Exercise 3.22]{hartshorne1977algebraic}) combined with the definition of $X_i$ there is a string of (in)equalities
\begin{align*}
\dim K_i &\geq \dim X_i - \dim V^{i+2} \\
&= \dim K_{i-1} + \dim E^i_{a-ia} - \dim V^{i+2}\\
&\geq \dim K_{i-1} + \dim E^i_{a-ia} - \dim E^{i+2}_{2a-ia}.
\end{align*}
Setting $i = \ell$ and inducting downward, it follows that
$$\dim \underbrace{X_\ell}_{= \calX^\bfF_a} \geq \sum_{i=2}^\ell \dim E^i_{a-ia} - \sum_{i=4} E^i_{2a-ia}$$
$$=\sum_{i=2}^\ell\dim_\bbk H^i (\End_S^\bullet (\bfF))_{a-ia} - \sum_{i=4}^\ell\dim_\bbk H^i (\End_S^\bullet (\bfF))_{2a-ia}.$$
This establishes the lower bound.

Finally, it remains to show that if $\dim \calX^\bfF_a = 0$, then in fact $\calX^\bfF_a = \{ \fold^a (\bfF ) \}$ is a single point. A priori, the assumption that $\dim \calX^\bfF_a = 0$ implies that $\calX^\bfF_a$ consists of a finite set of points. Suppose for sake of contradiction that $\calX^\bfF_a$ contains another point corresponding to an isomorphism class distinct from that of $\fold^a (\bfF)$, and denote the corresponding free flag by $F$. Choose any $\lambda \in \bbk^\times$. If the maps $\delta_2 , \dots , \delta_\ell$ denote the data defining the free flag $F$, then the data $\lambda \delta_2 , \lambda^2 \delta_3, \dots , \lambda^{\ell-1} \delta_\ell$ defines another free flag, denoted $F_\lambda$, that is also anchored on $\bfF$. 

We claim that $F_{\lambda} \not\cong F$ if $\lambda$ is not a root of the polynomials $x^r - 1$ for $1 \leq r \leq \ell-1$. Notice that this gives an immediate contradiction since the fact that $\bbk$ is infinite implies that $\calX^\bfF_a$ must have infinitely many points (note that $\bbk$ is infinite by the algebraic closure assumption). To prove the claim, assume for sake of contradiction that $F_\lambda \cong F$ for $\lambda$ as above. Denote by $F_\lambda^t$ (resp. $F^t$) the deformed version of $F_\lambda$ (resp. $F$) obtained by rescaling the maps $\delta_i$ by powers of $t$. Then the free flags $F_\lambda^t$ and $F^t$ must also be isomorphic, and hence $F_\lambda^t \otimes_S S_i \cong F^t \otimes_S S_i$ for all $i \geq 1$.    

Choose $p := \min \{ i \geq 2 \mid \delta_i \neq 0 \}$. Observe that $p$ is well-defined by the assumption that $F$ is not in the isomorphism class of $\fold^a (\bfF)$. By Lemma \ref{lem:param}, the fact that $F_\lambda^t \otimes_S S_{p+1} \cong F^t \otimes_S S_{p+1}$ implies that the difference $(\lambda^{p-1} - 1)\delta_p$ is nullhomotopic. By the assumptions on $\lambda$ the scalar $\lambda^{p-1} - 1$ is nonzero, which implies that $\delta_p$ is nullhomotopic. Thus $F$ may be replaced with an isomorphic free flag satisfying $\delta_2 = \delta_3 = \cdots = \delta_p = 0$. Iterating this argument, it follows that $F$ has an isomorphic representative with $\delta_i = 0$ for all $i \geq 2$, implying that $F \cong \fold^a (\bfF)$; this contradiction yields the result.
\end{proof}

\begin{remark}
    It is worth mentioning that the assumption of algebraic closure for the field $\bbk$ is only necessary for invoking the equivalence of categories between the classical notion of a variety and reduced integral $\bbk$-schemes (which is induced by taking closed points). In general, the results of \ref{subsec:buildFreeFlag} are totally independent of any assumptions on the base field $\bbk$ and hence always give rise to at least an algebraic parametrization of free flags with a fixed anchor.
\end{remark}

In the case of a free flag anchored on a resolution, we can combine this with the results of Brown-Erman to deduce:

\begin{corollary}\label{cor:theMFRcase}
    Assume $\bbk$ is algebraically closed and let $M$ be any $S$-module. Then the set of quasi-isomorphism classes of degree $a$ differential modules with homology $M$ is parameterized by a nonempty variety $\calX^M_a$ whose dimension satisfies the following bounds:
    \[
\sum_{i=2}^\ell\dim_\bbk \ext_S^i (M,M)_{a-ia} - \sum_{i=4}^\ell\dim_\bbk \ext_S^i (M,M)_{2a-ia}\leq \dim \calX^M_a \leq \sum_{i=2}^\ell \dim_\bbk \ext_S^i (M,M)_{a-ia}.
\]
Moreover, $\dim \calX^M_a = 0$ if and only if the only quasi-isomorphism class of a differential module with homology $M$ is $M$ itself (viewed as a differential module with $0$ differential).
\end{corollary}

\begin{proof}
    By the Brown-Erman result (Theorem \ref{thm:danMikeThm}) the set of quasi-isomorphism classes of differential modules with fixed homology $M$ is equivalently the set of all isomorphism classes of free flags anchored on a minimal free resolution $\bfF$ of $M$. One only needs to check that the constructions of subsection \ref{subsec:buildFreeFlag} are independent of the resolution $\bfF$, but this is a consequence of Proposition \ref{prop:flagsAnchoredOnRes} combined with the independence of $\ext$ on the chosen free resolution of $M$.

    Assume now that $\dim \calX^M_a = 0$. Suppose $D$ is any differential module with $H(D) = M$ and let $F$ be any free flag resolution of $D$ anchored on a minimal free resolution $\bfF$ of $M$. By Theorem \ref{thm:varietyOfIsoClasses} the free flag $F$ is isomorphic to $\fold^a (\bfF)$, in which case $D$ is quasi-isomorphic to $\fold^a (\bfF)$. Since $\fold^a (\bfF)$ is evidently quasi-isomorphic to $M$, the result follows.
\end{proof}

It turns out that we can also obtain dimension bounds for arbitrary (not necessarily homogeneous) differential modules with finite length homology:

\begin{theorem}\label{thm:theUngradedCase}
Assume that $\bbk$ is algebraically closed and $\bfF$ has finite length homology. The set of isomorphism classes of free flags in $\DM (S)$ anchored on $\bfF$ is parameterized by a nonempty variety $\calX^\bfF$ whose dimension satisfies the bounds
\[
\dim_\bbk H^2 (\End_S (\bfF)) + \dim_\bbk H^3 (\End_S (\bfF)) \leq \dim \calX^\bfF \leq \sum_{i=2}^\ell \dim_\bbk H^i (\End_S (\bfF)).
\]
\end{theorem}

\begin{proof}
Observe first that the constructions and results of subsection \ref{subsec:buildFreeFlag} are inherently independent of the homogeneity assumption, in which case we may employ the same argument as in Theorem \ref{thm:varietyOfIsoClasses}. The finite length assumption ensures that all vector spaces involved are finite dimensional.
\end{proof}

\section{Examples and Applications}\label{sec:Applications}
In this section, we discuss some of the applications of our results including rigidity conditions for free complexes, total Betti numbers of differential modules, and matrix factorizations.

\subsection{Rigid resolutions}

In general, the quasi-isomorphism class of a degree $a$ differential module is only controlled up to some choice of free flag anchored on a resolution (this is the content of \cite{brown2021minimal}), but it is useful to know when one can guarantee that there are \emph{no} nontrivial free flags anchored on a given complex. This is precisely the notion of $a$-rigidity:

\begin{definition}\label{def:aRigid}
    A homogeneous complex $\bfF$ is $a$-\emph{rigid} if the only isomorphism class of a degree $a$ free flag anchored on $\bfF$ is the degree $a$ fold $\fold^a (\bfF)$. A homogeneous $S$-module $M$ is $a$-rigid if some (equivalently every) homogeneous free resolution of $M$ is $a$-rigid.
\end{definition}

\begin{remark}
    Another description of $a$-rigidity may be given as follows: consider the homology functor $H : \dm (S,a) \to S\text{-mod}$ which sends a given differential module $D$ to its homology $H(D)$. Then a given $S$-module $M$ is $a$-rigid if the ``fiber" above $M$ is (essentially) contained inside the category of chain complexes, where $\ch (S) \subset \dm (S,a)$ as in Remark \ref{rem:embeddingCh(S)}. 
\end{remark}

Our results also allow us to generalize a result previously proved in \cite{banks2022differential}, and in fact we can give a full characterization of $a$-rigid Artinian complete intersections: 

\begin{theorem}\label{thm:CIrigidity}
Let $S = \bbk [x_1 , \dots , x_n]$ ($n \geq 2$) be a standard-graded polynomial ring and $M = S/I$ be an Artinian complete intersection, where $I = (f_1 , \dots , f_n)$ is generated by a homogeneous regular sequence with degrees $\underline{d} := (d_1 \leq d_2 \leq \cdots \leq d_n)$ (where $d_i := \deg f_i$). Then:
\begin{enumerate}
    \item If $I$ is not generated by linear forms and 
    \[
    a < d_1+d_2+n - |\underline{d}| \quad\text{or}
    \quad a > d_{n-1} + d_n
    \]
    then $S/I$ is $a$-rigid.
    \item If $I$ is generated by linear forms, then $S/I$ is $a$-rigid for all $a \neq 2$.
\end{enumerate}
\end{theorem}

\begin{remark}
    The hypotheses of Theorem \ref{thm:CIrigidity} omit the case $n=1$ for the simple reason that this case is trivial. If $n=1$, then a free flag is equivalently a complex, so $S/I$ is $a$-rigid for all $a \in \bbz$. Likewise, if $n=2$ then the isomorphism classes of free flags are in direct bijection with the Ext classes $\ext^2 (S/I, S/I)$, so the statement really only needs to be proved for $n \geq 3$. 
\end{remark}

\begin{remark}
    The quantity $|\underline{d}| - n$ is the \emph{socle degree} of the complete intersection $S/I$, typically denoted $\socdeg (S/I)$. This means that the bounds of Theorem \ref{thm:CIrigidity}$(i)$ may be instead written as
    $$a <  d_1 + d_2 - \socdeg (S/I ) \quad \text{or} \quad  a > d_{n-1} + d_n \}.$$
\end{remark}

A quick example illustrating the statement of Theorem \ref{thm:CIrigidity} may be useful before the proof:

\begin{example}
    Suppose that $I$ is a complete intersection with degree sequence $\underline{d} = (2,2,5,7,9)$ (where $S/I$ is Artinian). Then
    $$\socdeg(S/I) = |\underline{d}| - n = 25 - 5 = 20,$$
    $$d_1+d_2 - \socdeg(I) = 4-20 = -16, \quad \text{and} \quad d_{n-1}+d_n = 7 + 9 = 16.$$
    Thus $S/I$ is $a$-rigid for all 
    $a<-16$ and all $a>16$.
\end{example}

\begin{proof}[Proof of Theorem \ref{thm:CIrigidity}]
    We will actually prove a much more precise bound, then specialize to the case at hand. It is well-known that for a finite length complete intersection $S/I$ over a polynomial ring, the maximal nonzero degree of $S/I$ is precisely $|\underline{d}| - n$. Moreover, there is an isomorphism
    $$\ext^i_S (S/I,S/I) \cong S/I \otimes_S \bigwedge^\bullet V,$$
    where $V = \bigoplus_{i=1}^n Se_i$ is a free $S$-module on basis elements $e_i$ assigned internal degree $\deg (e_i) = -d_i$. For every $1 \leq i \leq n$, there is a direct sum decomposition
    $$\bigwedge^i V = \bigoplus_{1 \leq j_1 < \cdots < j_i \leq n} S e_{j_1} \wedge \cdots \wedge e_{j_i},$$
    in which case
    $$\ext^i_S (S/I,S/I) = \bigoplus_{1 \leq j_1 < \cdots < j_i \leq n} S/I e_{j_1} \wedge \cdots \wedge e_{j_i}.$$
    Let $\ell_i$ (respectively $r_i$) be the minimal (respectively maximal) internal degree of $\bigwedge^i V$. In order for $\ext^i_S (S/I, S/I)_{a-ia}$ to be $0$, there must be strict inequalities
    \begin{equation} a-ai > |\underline{d}| - n + r_i \quad \text{or} \quad a - ai < \ell_i.\end{equation}\label{eqn:theIneqs}

    \textbf{Proof of (i):} Notice that $r_i \leq -(i-1) d_1 - d_2$ and $\ell_i \geq -(i-1) d_n - d_{n-1}$, in which case substituting these into the inequalities (\ref{eqn:theIneqs}) implies that $S/I$ will be $a$-rigid if the following inequalities hold for all $2 \leq i \leq n$:
    $$a - ai > |\underline{d}| - n - (i-1)d_1 - d_2 \quad \text{or} \quad a - ai < -(i-1) d_n - d_{n-1}.$$
    Rearranging yields the inequalities
    $$a < d_1 + \frac{n - |\underline{d}| + d_2}{i-1}, \quad a >  d_n + \frac{ d_{n-1}}{i-1}.$$
    The quantity $d_n +\frac{ d_{n-1}}{i-1}$ is evidently maximal for $i = 2$, yielding $a > d_n + d_{n-1}$. On the other hand, we claim that $n - |\underline{d}| + d_2 \leq 0$.

    To see this, suppose for sake of contradiction that $|\underline{d}| < n + d_2$. Since $ d_1 + (n-1)d_2 \leq |\underline{d}|$, it follows that $\frac{d_1}{n-2} + d_2 <  \frac{n}{n-2}$. If $n \geq 4$, this implies that $d_2 < 2$ and hence $d_1 = d_2 = 1$. Likewise, if $n = 3$ then $d_1 + d_2 <3$, again implying that $d_1 = d_2 = 1$ and hence $|\underline{d}| < n+1$.
    
    Since each $d_j > 0$, it follows that $|\underline{d}| = n$ and $d_j = 1$ for all $1 \leq j \leq n$. This contradicts the assumption that $I$ is not generated by linear forms, so $n - |\underline{d}| + d_2 < 0$ and hence there is an inequality 
    $$d_1 + \frac{n - |\underline{d}| + d_2}{i-1} \geq d_1+d_2 +n - |\underline{d}| .$$
    It follows that if $a < d_1+d_2 +n - |\underline{d}|$ or $a > d_{n-1}+d_n$, then $S/I$ is $a$-rigid.

    \textbf{Proof of (ii):} This proof is essentially identical but much simpler: if $I$ is generated by linear forms, then $\ext^i_S (S/I , S/I)$ is generated in internal degree $-i$ for all $2 \leq i \leq n$. This means that if $S/I$ is $a$-rigid for some $a$, there is an equality $a - ai = -i$ for some $2 \leq i \leq n$. Rearranging this yields the equality $i = \frac{a}{a-1}$; since $i$ must be an integer, it follows that $a = i =2$.
\end{proof}

As it turns out, we can upgrade Theorem \ref{thm:CIrigidity} to an equivalence:

\begin{corollary}\label{cor:backwardsCIRigidity}
    The statement of Theorem \ref{thm:CIrigidity} is an equivalence. More precisely: adopt notation and hypotheses as in Theorem \ref{thm:CIrigidity}. If
    \[
    d_1+d_2+n - |\underline{d}| \leq a \leq d_{n-1} + d_n
    \]
    then $S/I$ is not $a$-rigid.
\end{corollary}

\begin{proof}
    The proof follows by constructing an explicit free flag anchored on the Koszul complex $K_\bullet$ resolving $S/I$ that is not isomorphic to $\fold (K_\bullet)$ for every $a$ in the appropriate interval. To do this, it is necessary to make some initial observations. Note first that $(S/I)_j \neq 0$ for all $0 \leq j \leq \socdeg (S/I)$, so for each $j$ we may fix a nonzero element $m_j \in (S/I)_j$. With notation as in the proof of Theorem \ref{thm:CIrigidity}, consider the elements
    $$m_j e_{i} \wedge e_{i+1} \in S/I \otimes_S \bigwedge^2 V, \quad 0 \leq j \leq \socdeg(S/I), \ 1 \leq i \leq n-1.$$
    Notice that each of these elements squares to $0$, even when viewed in $\bigwedge^\bullet V$, and hence induce a well-defined differential module by choosing all higher structure maps (of degree $ > 2$) to be $0$. Moreover, notice that for all $1 \leq i \leq n-2$ there is an inequality
    \begin{equation}\label{eqn:theNextIneq} 
    d_i + d_{i+1} \geq d_{i+1} + d_{i+2} - \socdeg (S/I)\end{equation}
    since rearranging the above expression yields the inequality
    $$n - d_{i} \leq |\underline{d}| - d_{i+2}.$$
    The righthand side of this inequality is at least $n-1$, and since $d_i \geq 1$ we also know that the lefthand side is at most $n-1$. With all of these ingredients, the proof concludes as follows: as $i$ and $j$ range through the values $1 , \dots , n$ and $0 , \dots , \socdeg(S/I)$, respectively, the elements $m_j e_{i} \wedge e_{i+1} \in \ext^2_S (S/I,S/I)_{-d_i-d_{i+1} + j}$ induce degree $a$ differential modules for all  $d_1 + d_2 - \socdeg(S/I) \leq a \leq d_{n-1} + d_n$. 
    The fact that $a$ must range through all values in this interval follows from the inequality (\ref{eqn:theNextIneq}) along with a straightforward induction on $i$.

    Finally, notice that each of the differential modules induced by the elements $m_j e_i \wedge e_{i+1}$ are not isomorphic to a fold of the Koszul complex, since Theorem 3.6 of \cite{banks2022differential} would imply that $m_j \in I$, a contradiction to the fact that $m_j$ represents a nonzero element in $(S/I)_j$ by selection.
\end{proof}

The following example should help to illustrate the proof of Corollary \ref{cor:backwardsCIRigidity} more clearly:

\begin{example}
    Let $S = \bbk [x_1 , x_2 , x_3]$ and $I = (x_1^2,x_2^2 , x_3^3)$. The socle degree of $S/I$ is $4$, in which case Corollary \ref{cor:backwardsCIRigidity} guarantees that we should be able to construct homogeneous free flags of degrees $\{ 0,1 , 2 , 3 , 4,5 \}$. To do this, choose elements:
    $$1 \in (S/I)_0, \quad x_1 \in (S/I)_1 , \quad x_1x_2 \in (S/I)_2 , \quad x_1x_2x_3 \in (S/I)_3, \quad x_1x_2x_3^2 \in (S/I)_4$$
    (note that the choice of these elements is arbitrary -- any nonzero choices will do). We then consider the collection of elements in $\ext^2_S (S/I,S/I)$:
    $$\underbrace{x_1 x_2 x_3^2 e_1 \wedge e_2}_{\text{degree} \ 0}, \quad \underbrace{x_1 x_2 x_3 e_1 \wedge e_2}_{\text{degree} \ -1}, \quad \underbrace{x_1 x_2 e_1 \wedge e_2}_{\text{degree} \ -2}, \quad \underbrace{x_1 e_1 \wedge e_2}_{\text{degree} \ -3} , \quad \underbrace{e_1 \wedge e_2}_{\text{degree} \ -4},$$
    $$\underbrace{x_1 x_2 x_3^2 e_2 \wedge e_3}_{\text{degree} \ -1}, \quad \underbrace{x_1 x_2 x_3 e_2 \wedge e_3}_{\text{degree} \ -2}, \quad \underbrace{x_1 x_2 e_2 \wedge e_3}_{\text{degree} \ -3}, \quad \underbrace{x_1 e_2 \wedge e_3}_{\text{degree} \ -4} , \quad \underbrace{e_2 \wedge e_3}_{\text{degree} \ -5}.$$
    Notice that there is redundancy in the degrees of the differential modules induced by the Ext elements in this collection, but they still range through \emph{all} values in the interval $\{ 0 , 1 , 2 , 3 , 4 , 5 \}$.
\end{example}

\begin{remark}
    The construction of Corollary \ref{cor:backwardsCIRigidity} actually generalizes the construction of \cite[Proposition 3.8]{banks2022differential}, since the differential modules constructed there are minimizations of the free flag induced by choosing $e_1 \wedge e_2 \in \ext^2_S (\bbk , \bbk )_2$. 
\end{remark}

It is not difficult to see that the construction of Corollary \ref{cor:backwardsCIRigidity} does not yield \emph{all} possible free flags of a given degree, it is only a sufficiency result that guarantees nonrigidity within a given degree range. The following example helps illustrate this:

\begin{example}
    Let $S = \bbk [x_1 , \dots , x_4]$ and $I = (x_1^2 , \dots , x_4^2)$, and assume that $\bbk$ has odd characteristic. Consider the elements
    $$f_1 := x_1 e_1 + x_2 e_2 + x_3 e_3 + x_4 e_4, \quad \text{and}$$
    $$f_2 := x_1x_2x_3x_4 (e_1 \wedge e_2 + e_3 \wedge e_4 ) \in \bigwedge^2 V.$$
    Notice that $f_2$ squares to $0$ after descending to the Ext algebra, but it does \emph{not} square to $0$ when viewed as an element in the exterior algebra $\bigwedge^\bullet V$, since
    $$\left( x_1x_2x_3x_4 (e_1 \wedge e_2 + e_3 \wedge e_4 ) \right)^2 = 2 x_1^2 x_2^2 x_3^2 x_4^2 e_1 \wedge e_2 \wedge e_3 \wedge e_4 \neq 0.$$
    Define
    $$f_3 := x_1x_2x_3x_4 \left(x_2 x_3 x_4 e_2 \wedge e_3 \wedge e_4   - x_1x_3x_4 e_1 \wedge e_3 \wedge e_4 + x_1x_2 x_4 e_1 \wedge e_2 \wedge e_4 - x_1 x_2 x_3 e_1 \wedge e_2 \wedge e_3\right)$$
    and notice that $f_1 \cdot f_3 - f_3 \cdot f_1 = f_2^2$. Since $f_2 \cdot f_3 + f_3 \cdot f_2 = 0$, it follows that the elements $f_1 , f_2,$ and $f_3$ induce a well-defined Koszul differential module as in \cite[Proposition 4.2]{banks2022differential}, and moreover this differential module is homogeneous of degree $0$. It is not isomorphic to any of the free flags of Corollary \ref{cor:backwardsCIRigidity}, however.
\end{example}

Next, we consider the problem of understanding free flags anchored on complexes with finite length homology. Free flags anchored on arbitrary complexes are a little more subtle than being anchored on a free resolution, since we do not have complete control over how the homology changes as the higher structure maps are added iteratively. However, we do at least know that the isomorphism classes are parametrized by modules concentrated in only finitely many internal degrees. This immediately yields:

\begin{theorem}\label{thm:finiteLengthAsymptotics}
Let $\bfF$ be any finite length complex of free $S$-modules with finite length homology. Then $\bfF$ is $a$-rigid for $a \gg 0$.
\end{theorem}

\begin{proof}
The isomorphism classes are determined by the size of the cohomology of the endomorphism complex $\End_S^\bullet ( \bfF)$. If $\bfF$ has finite length homology, the cohomology of $\End_S^\bullet ( \bfF)$ is $0$ in sufficiently large degrees.
\end{proof}

\begin{remark}
    The statement of Theorem \ref{thm:finiteLengthAsymptotics} may be interpreted as a statement on the asymptotic behavior of free flags anchored on finite length complexes. In particular, it implies that ``interesting" homogeneous deformations of homogeneous complexes with finite length homology are exceptional, in the sense that there is only a finite window where such deformations may occur. 
\end{remark}





\subsection{Differential modules with small rank}


In this section, we revisit the Rank Conjectures of Buchsbaum--Eisenbud--Horrocks \cite{buchsbaum1977algebra, HARTSHORNE1979117}, Halperin \cite{Halperin1985AspectsOT} and Carlsson \cite{10.1007/BFb0072815}. These conjectures, concerning lower bounds on the Betti numbers of free resolutions and toral rank of certain manifolds and CW complexes, formed part of the motivation for \cite{avramov2007class}. In searching for a unifying conjecture that would imply both the algebraic and topological rank conjectures, they posited that conjectures on total rank of free complexes should also hold for differential modules admitting a free flag structure. While Walker proved the Total Rank Conjecture for minimal free resolutions (outside of characteristic 2) \cite{walker2017total}, the more general conjecture was shown to be false outside of characteristic $2$ even for \emph{complexes} by Iyengar-Walker \cite{iyengar2018examples}. In light of these results, we are curious about where precisely counterexamples to the more general conjecture can occur. We explore this question by relating the total Betti number of a differential module to the total Betti numbers of its homology. To do this, we introduce the notion of \emph{Betti-deficiency}, which can be viewed as a twist on $a$-rigidity. First, recall the definition of a Betti number for graded differential modules:

\begin{definition}
    The \emph{$j\th$ Betti number} $\beta_j (D)$ of a $\bbZ$-graded differential module $D$ is defined as
    $$\beta_j (D) := \dim_\bbk H(F \otimes^{DM}_S \bbk)_j,$$
    where $F$ is any minimal free resolution of $D$ (as in \cite{brown2021minimal}). The \emph{total Betti number} is \[
    \beta(D) = \sum_{j\in\bbZ} \beta_j(D).
    \] 
    A differential module $D$ is \emph{Betti-deficient} is there is a strict inequality
    $$\beta (D) < \beta (H(D)),$$
    where $H(D)$ is being viewed as a differential module with $0$ endomorphism.
\end{definition}

The $j\th$ Betti number of $D$ counts the number of generators in degree $j$ of a minimal free resolution of $D$. Likewise, the total Betti number counts the total number of generators (i.e. the total rank) of a minimal free resolution of $D$.

\begin{remark}
    As noted in \cite{brown2021minimal}, Theorem \ref{thm:danMikeThm} implies that there is always an inequality
    $$\beta (D) \leq \beta (H(D)),$$
    so Betti-deficient differential modules are the cases for which this inequality is strict. 
\end{remark}

\begin{remark}
    Notice that if $M$ is an $a$-rigid $S$-module, then there are \emph{no} Betti-deficient differential modules with homology $M$. This is not an equivalence, however, and makes the notion of Betti-deficiency a slightly more subtle property to study.
\end{remark}

 Using the fact that any free flag resolution $F$ of a differential module $D$ may be chosen to be anchored on a minimal free resolution of $H(D)$, this means that a differential module is Betti-deficient if and only if the matrix representation of the endomorphism $d^F$ has entries residing in the field $\bbk$. On the other hand, we know that the higher off-diagonal maps are equivalently described as elements of the endomorphism complex associated to the anchor of $F$. Combining these two facts yields the following:

\begin{theorem}\label{thm:BettiSlope}
    Let $M$ be an $S$-module with finite projective dimension. 
If there exists a Betti-deficient differential module $D\in\DM(S,a)$ with $H(D) = M$, then the Betti table for $M$ must have two nonzero entries (in nonadjacent columns) lying on a line of slope $1-a + \frac{a}{j}$ for some $j\geq 2$.
\end{theorem}


\begin{proof}
   Let $M$ have minimal free resolution 
   \[
   \bfF : F_0\leftarrow F_1\leftarrow \cdots\leftarrow F_n
   \]
   where $F_i = \bigoplus_{k\in\bbZ}S(-k)^{\beta_{i,k}}$. 
   We will use the fact that the total Betti number of a differential module $(D,\del)$ is equal to the number of generators in a minimal free resolution of $D$, and that by \cite[Theorems 3.2 and 4.2]{brown2021minimal} we can assume that our minimal free resolution is a direct summand of an anchored free flag resolution of $D$. Using the minimization procedure in \cite[Proposition 4.1]{brown2021minimal}, $D$ has a minimal free resolution that is strictly smaller than its anchored free flag resolution exactly when the matrix representation of $\del$ contains a unit. Such a unit corresponds to a degree 0 map 

\begin{center}
\begin{tikzcd}[row sep=small]
F_i(ia) \arrow[rr] \arrow[d, equal] &  & F_{i-j}((i-j+1)a) \arrow[d, equal] \\
{\bigoplus_k S(ia-k)^{\beta_{i,k}}}               &  & {\bigoplus_\ell S((i-j+1)a-\ell)^{\beta_{i-j,\ell}}}         
\end{tikzcd}
\end{center}

where $j\geq 2$. This means that for some $i,j,k,\ell$ with $\beta_{i,k}, \beta_{i-j,\ell} \neq 0$, there is an equality 
\begin{align*}
ia-k &= (i-j+1)a-\ell \\
\ell - k &=(1-j)a.
\end{align*}

In the Betti table for $M$, $\beta_{i,k}$ is located in column $i$ and row $k-i$, and $\beta_{i-j,\ell}$ is located in column $i-j$ and row $\ell-(i-j)$, so $\beta_{i,k}$ and $\beta_{i-j,\ell}$ lie on a line of slope
\begin{align*}
    \frac{\ell-(i-j)-(k-i)}{j} = \frac{\ell-k+j}{j} = \frac{(1-j)a+j}{j} = 1-a+\frac{a}{j}.
\end{align*}
\end{proof}

To more concretely illustrate the idea of Theorem \ref{thm:BettiSlope} and its proof, suppose that $M$ has minimal free resolution 
\[
F_0\xleftarrow{\del_{1,0}} F_1\xleftarrow{\del_{2,1}} F_2\xleftarrow{\del_{3,2}} F_3
\]
and let $D$ be a degree $a$ differential module with homology $M$. If $D$ is Betti-deficient, then the matrix representing the differential of $D$'s minimal free resolution contains a unit. Suppose for example that the unit is in the red block in the matrix below, which we can visualize via the augmented complex on the right.
\begin{figure}[h]
\begin{minipage}[b]{0.22\textwidth}
    \vspace{20pt}
    $\begin{pmatrix}
        0 & \del_{1,0} & \del_{2,0} & \del_{3,0} \\
        0 & 0 & \del_{2,1} & \textcolor{red}{\del_{3,1}} \\
        0 & 0 & 0 & \del_{3,2}\\
        0 & 0 & 0 & 0
    \end{pmatrix}$
\end{minipage}
\begin{minipage}[b]{0.45\textwidth}
    \begin{tikzcd}[column sep=tiny]
F_0 \arrow[d, equal] & F_1(a) \arrow[l, "\del_{1,0}"] \arrow[d, equal] & F_2(2a) \arrow[l, "\del_{2,1}"]\arrow[d, equal] \arrow[ll, bend right, shift right=.75ex, "\del_{2,0}"] & F_3(3a) \arrow[l, "\del_{3,2}"] \arrow[lll, bend right, "\del_{3,0}"', shift right=1.5ex] \arrow[ll, bend right, red, crossing over, shift right=.75ex, "\del_{3,1}"]\arrow[d, equal] \\[-10pt] |[node font=\tiny]|
\bigoplus S(-j)^{\beta_{0,j}} & |[node font=\tiny]| \bigoplus S(a-j)^{\beta_{1,j}} & |[node font=\tiny]| \bigoplus S(2a-j)^{\beta_{2,j}} & |[node font=\tiny]| \bigoplus S(3a-j)^{\beta_{3,j}}
\end{tikzcd}
\end{minipage}
\end{figure}

The presence of a unit entry in the degree $a$ map $\del_{3,1}$ means that there is a summand of $F_1(a)$ and a summand of $F_3(3a)$ that are generated in the degrees that differ by $a$, i.e some $\beta_{1,j}, \beta_{3,k}\neq 0$ for which $(3a-k)-(a-j) = a$. If, for example, $a = 1$ one pair that would satisfy this is $\beta_{1,3}$ and $\beta_{3,4}$, in red in the Betti table below. The slope of line connecting these two entries is $\frac{1}{2} = 1-\frac{1}{1} + \frac{1}{2}$ as in the theorem.

\begin{figure}[h]
\begin{minipage}[b]{0.3\textwidth}
    \begin{tabular}{r|c c c c}
         &$i: $ 0 & 1 & 2 & 3  \\
        \hline $j-i$ : 0 & $\beta_{0,0}$ & $\beta_{1,1}$ & $\beta_{2,2}$ & $\beta_{3,3}$\\
         1 & $\beta_{0,1}$ & $\beta_{1,2}$ & $\beta_{2,3}$ & \textcolor{red}{$\beta_{3,4}$}\\
         2 & $\beta_{0,2}$ & \textcolor{red}{$\beta_{1,3}$} & $\beta_{2,4}$ & $\beta_{3,5}$\\
         3 & $\beta_{0,3}$ & $\beta_{1,4}$ & $\beta_{2,5}$ & $\beta_{3,6}$\\
        \end{tabular}
\end{minipage}
\end{figure}

In cases where we have a good understanding of the shape of the Betti table, we may obtain concrete numerical conditions for when a differential module may be Betti-deficient. For complete intersections for instance, we can write down a criterion in terms of the degrees of the generators of the ideal:

\begin{corollary}
    Let $I\subset S$ be minimally generated by a regular sequence $(f_1, \ldots, f_\ell)$ where $f_i$ has degree $d_i$. If there is a Betti-deficient differential module in $\DM(S,a)$ with homology $S/I$, then there are subsets $J_1, J_2\subseteq \{1, \ldots, \ell\}$ and $2\leq j\leq\ell$ satisfying $|J_2| = |J_1|+j$ where 
    \[
\sum_{r\in J_1}d_r - \sum_{r\in J_2}d_r = a-aj.
    \]
\end{corollary}

\begin{proof}

The degrees of the $i\th$ syzygies of $S/I$ are $\sum_{r\in J}d_r$ where $|J| = i$, which means that in the $i\th$ column of the Betti table the nonzero entries are located in rows $\left(\sum_{r\in J}d_r\right) - i$ for all subsets $J\subset \{1, \ldots, \ell\}$ with $|J|=i$. Thus the slope of the line connecting two nonzero entries in columns $i,i+j$ for $j\geq 2$ is equal to
\[
\frac{\left(\sum_{r\in J_1} d_r\right) - i - \left(\sum_{r\in J_2} d_r\right) + (i+j)}{j}
\]
where $J_1, J_2\subseteq \{1, \ldots, \ell\}$, $|J_1| = i$, and $|J_2| = i+j$. Setting this equal to the slope given in Theorem \ref{thm:BettiSlope} and simplifying yields the result.

\end{proof}

Another class of Betti tables of which we have a good understanding are the \emph{pure} Betti tables---those with only a single nonzero entry in each column.

\begin{corollary}
    Let $D\in \DM(S,a)$ be a differential module with homology $M$, where $M$ has a pure resolution with degree sequence $(d_0, \ldots, d_\ell)$. If $M$ is Betti-deficient, then 
    \[
    a = \frac{d_{i-j}-d_i}{1-j}
    \]
    for some $2\leq j\leq i \leq \ell$.
\end{corollary}

\begin{proof}
    The Betti table of $M$ has nonzero entries $\beta_{i,d_i}$ located in the $i\th$ column and $(d_i-i)\th$ row. As in Theorem \ref{thm:BettiSlope}, if $M$ is Betti-deficient then there are $\beta_{i,d_i}, \beta_{i-j,d_{i-j}}\neq 0$ with $j\geq 2$ satisfying
    \begin{align*}
ia-d_i & = (i-j+1)a - d_{i-j}\\
d_{i-j}-d_i &=(1-j)a\\
a &=\frac{d_{i-j}-d_i}{1-j}.
    \end{align*}
\end{proof}

Setting $a = 0$ and recalling that the degree sequence for a pure resolution satisfies $d_0<d_1<\cdots <d_\ell$ immediately yields the following:

\begin{corollary}
    Let $D\in \DM(S,0)$ be a differential module with homology $M$, where the graded minimal free resolution of $M$ is pure. Then the total Betti number of $D$ is equal to the sum of the Betti numbers of $M$ (I.e. there are no Betti-deficient degree 0 differential modules with pure homology).
\end{corollary}

\subsection{Beyond free flags}

We conclude with a curious observation connecting the idea of ``higher maps" of a free flag with the notion of ``systems of higher homotopies" coming from the matrix factorization literature.

Notice that the sequence of maps $\delta_1 , \dots , \delta_n$ in the definition of a free flag must satisfy the following equations in order for the endomorphism of the free flag to square to $0$:
$$\sum_{j=1}^{i-1} \delta_j \circ \delta_{i-j} = 0.$$
One can abstract these equations to yield different types of differential modules:

\begin{example}
Let $X$ denote a generic $5 \times 5$ skew-symmetric matrix and $R = k[X]$ its coordinate ring, where $k$ is any field. Recall that for an indexing set $I = (i_1 < \cdots < i_\ell)$, the notation $\pf_I (X)$ denotes the pfaffian of the submatrix of $X$ formed by deleting rows and columns $i_1 , \dots , i_\ell$ from $X$. The ideal of submaximal pfaffians $\pf (X)$ is a grade $3$ Gorenstein ideal with minimal free resolutions of the form
$$F: \quad 0 \to R = F_3 \xrightarrow{d_1^t}  R^5 = F_2 \xrightarrow{X} R^5 = F_1 \xrightarrow{d_1} R \to 0, \quad \textrm{where}$$
$$d_1 = \begin{pmatrix} \pf_1 (X) & - \pf_2 (X) & \dots & \pf_5 (X) \\ \end{pmatrix}.$$ 
Let $e_1 , \dots , e_5$, $f_1, \dots , f_5$, and $g$ denote bases of $F_1$, $F_2$, and $F_3$, respectively. The complex $F$ admits the structure of a graded-commutative DG-algebra extending the standard $S$-module structure with products (see \cite[Example 2.1.3]{avramov1998infinite})
$$e_i e_j = - e_j e_i = \sum_{k \neq i, j} (-1)^{i+j+k} \pf_{ijk} (X) f_k, \quad (i<j)$$
$$e_i f_j = f_j e_i = \delta_{ij} g.$$
Consider the map that multiplies on the left by $e_1$, denoted $\ell_{e_1}$. The matrix representation of $\ell_{e_1}$ restricted to each graded component may be computed using the above product formula as so:

\begin{minipage}{2.5cm}
    \[
\begin{pmatrix} 1 \\ 0 \\ 0 \\ 0 \\ 0 \end{pmatrix} \colon R^1 \to R^5,
    \]
\end{minipage}
\begin{minipage}{8cm}
    \[
\begin{pmatrix}
0 & 0 & 0 & 0 & 0 \\
0 & 0 & x_{45} & -x_{35} & x_{34} \\
0 & -x_{45} & 0 & x_{25} & - x_{24} \\
0 & x_{35} & - x_{25} & 0 & x_{23} \\
0 & -x_{34} & x_{24} & -x_{23} & 0 \\
\end{pmatrix} \colon R^5 \to R^5,
    \]
\end{minipage}
\begin{minipage}{4cm}
\vspace{-5pt}
\[
\begin{pmatrix} 1 & 0 & 0 & 0 & 0\\ \end{pmatrix}  \colon R^5 \to R^1.
\]
\end{minipage}\\

Consider the following block matrix:
$$d^D := \begin{pmatrix} 0 & d_1 & 0 & 0 \\
\ell_{e_1} & 0 & X & 0 \\
0 & -\ell_{e_1} & 0 & d_1^t \\
0 & 0 & \ell_{e_1} & 0 \\
\end{pmatrix}.$$
One can verify by direct computation that ${d^D}^2 = \pf_1 (X) \textrm{id}_F$. In other words, the same process for constructing a free flag has yielded a ``curved" (in the sense of, for instance, \cite{keller2010non}) differential module. Using the induced $\bbz / 2 \bbz$-grading to split $F$ into its odd and even parts yields a nonminimal matrix factorization of $\pf_1 (X)$. In the language of matrix factorizations, the maps $\delta_2 , \dots , \delta_n$ would be called a system of higher homotopies (in this case, there is only a single higher homotopy induced by left multiplication by $e_1$).
\end{example}

\bibliographystyle{amsalpha}
\bibliography{biblio}
\addcontentsline{toc}{section}{Bibliography}

\end{document}